\RequirePackage{fix-cm}
\documentclass[smallcondensed]{svjour3}     
\smartqed  
\usepackage{graphicx}
\usepackage{amsmath,amssymb}
\usepackage{theoremref}
\usepackage{cite}
\usepackage{url}
\usepackage{hyperref}

\newtheorem{thm}{Theorem}
\newtheorem{cor}{Corollary}
\newtheorem{rmk}{Remark}
\newtheorem{lem}{Lemma}

\newcommand{\C}{\mathbb C}
\newcommand{\Q}{\mathbb Q}
\newcommand{\Z}{\mathbb Z}
\newcommand{\N}{\mathbb N}
\newcommand{\R}{\mathbb R}

\begin{document}

\title{Sign changes of a product of Dirichlet characters and Fourier coefficients of Hecke eigenforms}


\titlerunning{Sign changes of a product}        

\author{Soufiane Mezroui}


\institute{Soufiane Mezroui \at
           LabTIC,\\
           SIC Department,\\
           ENSAT,\\
           Abdelmalek Essaadi University,\\
           Tangier, Morocco\\
           \email{mezroui.soufiane@yahoo.fr}
           }

\maketitle

\begin{abstract}
Let $f\in S_k(\Gamma_{0}(N))$ be a normalized Hecke eigenform of even integral weight $k$ and level $N$. Let $j\ge1$ be a positive integer. We prove that for almost all primes $p$, $p\nmid N$, and for all characters $\chi_{0}=\pm 1\pmod N$, the sequence $\left(\chi_{0}(p^{nj})a(p^{nj})\right)_{n\in\N}$ has infinitely many sign changes.
We also obtain a similar result for the sequence $\left(a(p^{j(1+2n)})\right)_{n\in\N}$ when $j$ is odd. 
\keywords{Sign change \and Fourier coefficients\and Cusp forms\and Dirichlet series}
\subclass{11F03\and 11F30\and 11F37}
\end{abstract}

\section{Introduction}

Let $k,N\in\N$ be integers. Throughout the paper, $S_k(\Gamma_{0}(N),\chi)$ denotes the space of cusp forms of weight $k$ and level $N$, with Dirichlet character $\chi\pmod N$. When $k$ is even and $\chi=1$, the trivial character modulo $N$, we denote $S_k(\Gamma_{0}(N),1)=S_k(\Gamma_{0}(N))$. If in addition $N = 1$, we abbreviate notation with $S_k$. 

In \cite{meher}, it has been shown that for every normalized Hecke eigenform $f$ of even integral weight $k$ on the modular group $SL_2(\mathbb{Z})$ with Fourier coefficients $a(n)$ ($n \geq 1$), each sequence $(a(n^{j}))_{n\geq 1}$ for $j\in\{2, 3, 4\}$ has infinitely many sign changes. The proof of this uses suitable estimates of the sums
$$
\sum_{n\leq x}\lambda(n^j)\text{ and }\sum_{n\leq x}\lambda^2(n^j),
$$
where $\lambda(n)$ is given by $\lambda(n)=\frac{a(n)}{n^{(k-1)/2}}$. 

Recently, Kohnen and Martin showed, in \cite{kohnen14}, that if $j$ is a positive integer then for almost all primes $p$ the sequence $(a(p^{jn}))_{n\geq 0}$ has infinitely many sign changes. The proof requires the use of Landau's theorem and suitable computations applied to the Dirichlet series
$$
\sum_{n\geq 0}a(p^{j n})p^{-jns}.
$$
In this work we extend the results of \cite{kohnen14} to normalized Hecke eigenforms of even integral weight $k$ and level $N$. Furthermore, we will show that the sequence $(a(p^{j(1+2n)}))_{n\geq 0}$ has infinitely many sign changes. More precisely, our first main theorem is the following.

\begin{thm}\thlabel{thm2}
Let $f\in S_k(\Gamma_{0}(N),\chi)$ be a normalized Hecke eigenform of even integral weight $k$ and level $N$, with Dirichlet character $\chi$. Let $\chi_{0}\pmod N$  be a Dirichlet character satisfying $\chi_{0}^{2}=\chi$. Let
$$
f(z)=\sum_{n\ge 1}a(n)e(nz),
$$
be the Fourier expansion of $f$ at $\infty$. Let $j\ge1$ be an integer. Then for almost all primes $p$, $p\nmid N$, the sequence $\left(\frac{a(p^{nj})}{\chi_{0}(p^{nj})}\right)_{n\in\N}$ has infinitely many sign changes.
\end{thm}

This result extends \cite[Theorem 2.1]{kohnen14}. Indeed, when $\chi=1$, we get the following result. 

\begin{cor} 
Let $f\in S_k(\Gamma_{0}(N))$ be a normalized Hecke eigenform of even integral weight $k$ and level $N$. Let $j\ge1$ be a positive integer. Then for almost all primes $p$, $p\nmid N$, and for all characters $\chi_{0}=\pm 1\pmod N$, the sequence $\left(\chi_{0}(p^{nj})a(p^{nj})\right)_{n\in\N}$ has infinitely many sign changes.
\end{cor}

Our second main theorem shows that the subsequence of $\left(a(p^{j(n)})\right)_{n\in\N}$, with odd indices, has infinitely many sign changes.

\begin{thm}\thlabel{thm3}
Let $f\in S_k(\Gamma_{0}(N))$ be a normalized Hecke eigenform of even integral weight $k$ and level $N$. Let $j\ge1$ be a positive integer such that $2\nmid j$. Then for almost all primes $p$, $p\nmid N$, the sequence $\left(a(p^{j(1+2n)})\right)_{n\in\N}$ has infinitely many sign changes.
\end{thm}

It should be noted that the proofs of these two theorems rely on Landau's theorem applied to the suitable Dirichlet series and Deligne's bound for the Fourier coefficients $a(n)$.

Let $f\in S_k(\Gamma_{0}(N),\chi)$ be a cusp form with Fourier coefficients $a(n)$, $n\geq 1$. Let $j\geq 1$  be any non-negative integer and $p$ a prime  number. In order to state the following theorem, we define the operator $T_j(p)$ acting on $S_k(\Gamma_{0}(N),\chi)$ by   

\begin{equation}
T_j(p)f(z)= \sum_{n\ge1}\left(a(p^{j}n)+p^{j(k-1)}\chi^{j}(p)a\left(\frac{n}{p^{j}}\right)\right)e(nz),\label{eq,4}
\end{equation}

with the convention $a(n/p^{j})=0$ if $p^j$ does not divides $n$. Notice that $T_0(p)=2$ and $T_1(p)=T(p)$ where $T(p)$ is the $p$-th classical Hecke operator. When $f\in S_k$, these operators are the same as those defined in \cite{kohnen14}, and it was shown in this case that the characteristic polynomial $P(T_{j}(p),X)$ of $T_{j}(p)$ on $S_k$ has rational coefficients.

\begin{thm}\thlabel{thm4}
Suppose that $P(T_{j}(p),X)$ is irreducible over $\mathbb{Q}$. Assume further that there are no different eigenvalues $\lambda_{1}$ and $\lambda_2$ of $T_{j}(p)$ such that $\lambda_1+\lambda_2=0$. Let $f\in S_k$ be a non zero cusp form of even integral weight $k$ with Fourier coefficients $a(n)$, $n\geq 1$. Let $j\ge1$ be a positive integer such that $2\nmid j$. Then for almost all primes $p$, $p\nmid N$, the sequence $\left(a(p^{j(1+2n)})\right)_{n\in\N}$ has infinitely many sign changes.
\end{thm}

Notice that when $j=1$ and $T_1(p)=T(p)$, the conjecture of Maeda says that $P(T(p),X)$ is irreducible over $\mathbb{Q}$. This conjecture is supported by some numerical results \cite{baba,Ahlgren,farmer}. 

\section{Proof of \texorpdfstring{\thref{thm2}}{Theorem 2}}

In this subsection, we prove \thref{thm2}. We begin with the following lemma.

\begin{lem}\thlabel{lem}
Let $p$ be a prime number and $j\ge1$ an integer. The following assertions hold.

\begin{enumerate}
\item $T_j(p)$ is a monic polynomial in $T(p)$ of degree $j$.
\item If $f\in S_k(\Gamma_{0}(N),\chi)$ is an eigenfunction of $T_j(p)$ with eigenvalue $\lambda_j(p)$, then

\begin{equation}
\sum_{n\ge0}\frac{a(p^{jn})}{\chi_{0}(p^{jn})}X^n=\dfrac{1}{1-\frac{\lambda_j(p)}{\chi_{0}(p^{j})}X+p^{j(k-1)}X^2}\label{eq.1}
\end{equation}
 where $a(n)$ denotes the $n$-th Fourier coefficient of $f$.
\end{enumerate}
\end{lem}

\begin{proof}[Proof of \thref{lem}]
\begin{enumerate}
\item  We see easily from \eqref{eq,4} that for all $j\ge 1$ one has
$$
T_{j+1}(p)=T_j(p)T(p)-p^{k-1}\chi(p)T_{j-1}(p),
$$
hence the result follows by recurrence on $j$.
\item Let $n\in\N$. We have
\begin{equation}
a(p^{j(n+1)})=\lambda_j(p)a(p^{jn})-p^{j(k-1)}\chi^j(p)a(p^{j(n-1)}),\label{eq,99}
\end{equation}
for all $j\ge1$, which can be deduced from \eqref{eq,4}. Therefore,

\begin{align*}
S&=\sum_{n\geq 0}\frac{a(p^{jn})}{\chi_{0}(p^{jn})}X^{n}\\
&=a(1)+\frac{a(p^{j})}{\chi_{0}(p^{j})}X+\sum_{n\geq 0}\frac{a(p^{j(n+2)})}{\chi_{0}(p^{j(n+2)})}X^{n+2}\\ 
&=a(1)+\frac{a(p^{j})}{\chi_{0}(p^{j})}X+\sum_{n\geq 0}
\frac{\lambda_{j}(p)}{\chi_{0}(p^{j})}\frac{a(p^{j(n+1)})}{\chi_{0}(p^{j(n+1)})}X^{n+2}\\
&\hspace{0.3 cm }-\sum_{n\geq 0}p^{j(k-1)}\frac{\chi^{j}(p)}{\chi_{0}^{2j}(p)}\frac{a(p^{jn})}{\chi_{0}(p^{jn})}X^{n+2}.
\end{align*}

Since $\chi=\chi_{0}^{2}$, then
$$
S=a(1)+\frac{a(p^{j})}{\chi_{0}(p^{j})}X+\frac{\lambda _{j}(p)}{\chi_{0}(p^{j})}(S-a(1))X-p^{j(k-1)}S X^{2}.
$$
Hence
$$
S=\dfrac{a(1)+\left(\frac{a(p^j)-a(1)\lambda_{j}(p)}{\chi_{0}(p^j)}\right)X}{1-\frac{\lambda_j(p)}{\chi_{0}(p^{j})}X+p^{j(k-1)}X^2}\cdot
$$
Replacing $n=0$ in \eqref{eq,99}, we obtain $a(p^j)=a(1)\lambda_{j}(p)=\lambda_{j}(p)$. This proves the Lemma.
\end{enumerate}
\end{proof}

\begin{proof}[Proof of \thref{thm2}]
Let $f\in S_k(\Gamma_{0}(N),\chi)$ be a normalized Hecke eigenform of even integral weight $k$ and level $N$, with Dirichlet character $\chi$. Let $\chi_{0}\pmod N$  be a Dirichlet character such that $\chi_{0}^{2}=\chi$. Let $j$ be an integer. It is well known that $\forall n\in\mathbb{N}$, $a(n)=\chi(n)\overline{a(n)}=\chi_{0}^{2}(n)\overline{a(n)}$. Let $p$ be a prime, $ p\nmid N$. Then $\chi_{0}(p^{nj})\neq 0$ and the above equation implies $a(p^{nj})=\chi_{0}^{2}(p^{nj})\overline{a(p^{nj})}$. Hence 
$$
\frac{a(p^{nj})}{\chi_{0}(p^{nj})}=\frac{\overline{a(p^{nj})}}{\overline{\chi_{0}(p^{nj})}},
$$ 
from which we obtain $\frac{a(p^{nj})}{\chi_{0}(p^{nj})}\in\mathbb{R}$. Suppose that the sequence $\left(\frac{a(p^{jn})}{\chi_{0}(a(p^{jn}))}\right)_{n\geq 0}$ does not have infinitely many sign
changes.

Applying Landau's theorem, we deduce that the Dirichlet series

\begin{equation}
\sum_{n\ge0}\frac{a(p^{jn})}{\chi_{0}(p^{jn})}p^{-jns}\quad (\Re(s)\gg 1),\label{eq:2}
\end{equation}
either has a pole on the real point of its line of convergence or must converges for all $s\in\C$. We will disprove the both assertions when $p$ is large. 

We start by considering the first case. Since $f$ is a normalized Hecke eigenform, we have $a(p^{n})=a(p)a(p^{n-1})-\chi(p)p^{k-1}a(p^{n-1})$ for all integers $n\in\mathbb{N}$. Taking this and applying the similar computations of \thref{lem}, we get 

\begin{equation}
P(X)=\sum_{n\ge0}\frac{a(p^n)}{\chi_{0}(p^{n})}X^n=\dfrac{1}{1-\frac{a(p)}{\chi_{0}(p)}X+p^{k-1}X^2}\cdot\label{eq.2}
\end{equation} 
The denominator of the right-hand side of \eqref{eq.2} factorizes as

\begin{equation}
1-\frac{a(p)}{\chi_{0}(p)}X+p^{k-1}X^2=(1-\alpha_pX)(1-\beta_pX),\label{eq:4}
\end{equation}
where

\begin{equation}
\alpha_{p},\beta_{p}=\frac{\frac{a(p)}{\chi_{0}(p)}\pm \sqrt{\left(\frac{a(p)}{\chi_{0}(p)}\right)^{2}-4p^{k-1}}}{2}\cdot\label{eq:08}
\end{equation}
Applying Deligne's bound, $\left(\frac{a(p)}{\chi_{0}(p)}\right)^{2}=\mid a(p)\mid ^{2}\leq 4p^{k-1}$, since $\left(\frac{a(p)}{\chi_{0}(p)}\right)\in\mathbb{R}$. We deduce that $\alpha_p$ and $\beta_p$ are complex conjugates numbers $\beta_{p}=\overline{\alpha_{p}}$.

Let $\zeta:=e^{2\pi i/j}$ be a primitive $j$-th root of unity and let $\nu\in\Z$. The following orthogonality relation
\begin{displaymath}
\sum_{\mu=0}^{j-1}\zeta^{\mu\ell}=
\begin{cases}
j, & \text{if $\ell\equiv 0$ {\rm (mod j),}}\\
0, & \text{if $\ell\not\equiv 0$ {\rm (mod j),}}
\end{cases}
\end{displaymath}
implies
$$
\sum_{n\ge0}\frac{a(p^{jn})}{\chi_{0}(p^{jn})}X^{jn}=\frac{1}{j}\sum_{\mu=0}^{j-1}P(\zeta^{\mu}X)\cdot
$$
Replacing $X=p^{-s}$ ( $s\in\C$ ), we get

\begin{equation}
\sum_{n\ge0}\frac{a(p^{jn})}{\chi_{0}(p^{nj})}p^{-jns}=\frac{1}{j}\sum_{\mu=0}^{j-1}\dfrac{1}{(1-\zeta^{\mu}\alpha_p p^{-s})(1-\zeta^{\mu}\beta_p p^{-s})}\quad (\Re(s)\gg1)\cdot\label{eq:5}
\end{equation}
Notice that using \eqref{eq:5},  the Dirichlet series 
$$
\sum_{n\ge0}\frac{a(p^{jn})}{\chi_{0}(p^{nj})}p^{-jns}
$$ 
can be meromorphicaly extended to the whole complex plane $\mathbb{C}$.

Suppose now that one of the denominators on the right-hand side of \eqref{eq:5} has a real zero, for example $\alpha_{p}\zeta^{\mu}\in\R$. Then $\alpha_p\zeta^{\mu}=\nu\in\R$. This implies $\overline{\alpha_p}\zeta^{-\mu}=\nu$, and using \eqref{eq:08} we get $\nu^2=|\alpha_p|^2= p^{k-1}$. Therefore $\nu=\pm p^{(k-1)/2}$. It follows that
$$
a(p)=(\alpha_p+\beta_p)\chi_{0}(p)=\pm p^{(k-1)/2}(\zeta^{-\mu}+\zeta^{\mu})\chi_{0}(p).
$$
We get the same result if we start with the condition that $\beta_p\zeta^\mu$ is real.

Suppose, for the sake of contradiction, there are infinitely many primes $p$ for which there are integers $\mu_p\pmod j$
such that

\begin{equation}
a(p)=\pm p^{(k-1)/2}(\zeta^{-\mu_p}+\zeta^{\mu_p})\chi_{0}(p).\label{eq:7}
\end{equation}
It is well known that 
$$
K_f:=\Q(\{a(p)\}_p),
$$
the subfield of $\C$ generated by all $a(p)$, where $p$ runs on primes, is a number field. Particularly, it is a finite extension of $\Q$. Therefore, the field $K_{f}(\zeta)$ is also a finite extension of $\Q$. Let $\mathbb{K}$ denote the field obtained by adjoining all $\chi_{0}(n)$, $\forall n\in\mathbb{N}$, to the field $K_{f}(\zeta)$. The field $\mathbb{K}$ is also a number field and particularly, a finite extension of $\Q$. From \eqref{eq:7} and since $k$ is even, we see
$$
\sqrt{p}\in\mathbb{K}.
$$
By our hypothesis, we conclude that there are infinitely many primes $p_1<p_2<p_3\dots$ satisfying
$$
\Q(\sqrt{p_1},\sqrt{p_2},\sqrt{p_3},\ldots)\subset \mathbb{K}.
$$
However, it is a classical fact that the degree of the extension 
$$
\Q(\sqrt{p_1},\sqrt{p_2},\sqrt{p_3},\dots)/\Q
$$
is infinite, which gives our contradiction. Consequently, we have proved that, for almost all primes $p$, the right-hand side of \eqref{eq:5} has no real poles.  

It remains to exclude the second case of Landau's theorem. Suppose that for a prime $p$, the series \eqref{eq:2} converges everywhere, and particularly, it is an entire function in $s$. By (1) of \thref{lem} we see that $f$ is an eigenfunction of $T_j(p)$. Let $\lambda_j(p)$ be the corresponding eigenvalue, hence from (2) of \thref{lem} we get 
\begin{equation}
\sum_{n\ge0}\frac{a(p^{jn})}{\chi_{0}(p^{jn})}X^{jn}=\dfrac{1}{1-\frac{\lambda_j(p)}{\chi_{0}(p^{j})}X^{j}+p^{j(k-1)}X^{2j}}\cdot\label{eeqq1}
\end{equation}
The denominator on the right-hand side is a polynomial in $X^j$ of degree $2$, hence
it is non-constant and so has zeros. Setting $X=p^{-s}$ to obtain a contradiction.

\end{proof}

\section{Proof of \texorpdfstring{\thref{thm3}}{Theorem 3}}

Assume the hypothesis of \thref{thm2}. We want  to compute the following sum
$$
S_{1}(X)=\sum_{n=0}^{\infty}\frac{a(p^{1+2 n})}{\chi_{0}(p^{1+2 n})}X^{1+2 n},
$$
By the same reasoning as in \eqref{eq:5} we have

\begin{align}
S_{0}(X)&=\sum_{n\ge0}\frac{a(p^{2 n})}{\chi_{0}(p^{2 n})}X^{2 n}=\frac{1}{2}\sum_{\mu=0}^{1}\dfrac{1}{(1-(-1)^{\mu}\alpha_p X)(1-(-1)^{\mu}\beta_p X)}\\
&=\frac{1+\alpha_{p}\beta_{p}X^{2}}{\left(1-\alpha_{p}^{2}X^{2}\right)\left(1-\beta_{p}^{2}X^{2}\right)}\cdot\label{eq,45}
\end{align}

Since $S_{1}(X)=P(X)-S_{0}(X)$, we obtain

\begin{equation}
S_{1}(X)=\frac{(\alpha_{p}+\beta_{p})X}{(1-\alpha_{p}^{2}X^{2})(1-\beta_{p}^{2}X^{2})}\cdot\label{eq,42}
\end{equation}

Now, let $j\ge1$ be an integer. Let $S_{1,j}$ denote the following sum
\begin{equation}
S_{1,j}(X)=\sum_{n=0}^{\infty}\frac{a(p^{j(1+2 n)})}{\chi_{0}(p^{j(1+2 n)})}X^{j(1+2 n)}\cdot\label{eq,52}
\end{equation}

Assume further that the integer $j\geq 1$ satisfy $(j,2)=1$. Once again, let $\zeta:=e^{2\pi i/j}$ be a primitive $j$-th root of unity and let $\nu\in\Z$. The orthogonality relation
\begin{displaymath}
\sum_{\mu=0}^{j-1}\zeta^{\mu\nu}=
\begin{cases}
j, & \text{if $\nu\equiv 0$ {\rm (mod j),}}\\
0, & \text{if $\nu\not\equiv 0$ {\rm (mod j),}}
\end{cases}
\end{displaymath}
implies

\begin{equation}
S_{1,j}(X)=\frac{1}{j}\sum_{\mu =0}^{j-1}S_{1}(\zeta^{\mu}X).\label{eq,51}
\end{equation}

\begin{proof}[Proof of \thref{thm3}]

Assume the hypothesis of \thref{thm2} and take , $2\nmid j$, $\chi=1$, $\chi_{0}=1$. Replacing this in \eqref{eq,52} to obtain

\begin{equation}
S_{1,j}(X)=\sum_{n=0}^{\infty}a(p^{j(1+2 n)})X^{j(1+2 n)}=\frac{1}{j}\sum_{\mu =0}^{j-1}S_{1}(\zeta^{\mu}X),\label{eq,53}
\end{equation}

where

\begin{equation}
S_{1}(\zeta^{\mu} X)=\frac{(\alpha_{p}+\beta_{p})\zeta^{\mu}X}{(1-\alpha_{p}^{2}\zeta^{2\mu}X^{2})(1-\beta_{p}^{2}\zeta^{2\mu}X^{2})}\cdot
\end{equation}

Replacing $X=p^{-s}$ ( $s\in\C$ ), we obtain

\begin{equation}
\sum_{n=0}^{\infty}a(p^{j(1+2 n)})\frac{1}{p^{s j(1+2 n)}}=\frac{\alpha_{p}+\beta_{p}}{j p^{s}}\sum_{\mu =0}^{j-1}\frac{\zeta^{\mu}}{(1-\alpha_{p}^{2}\zeta^{2\mu}\frac{1}{p^{2 s}})(1-\beta_{p}^{2}\zeta^{2\mu}\frac{1}{p^{2 s}})}\label{eq,54}\cdot
\end{equation}

Using this formula, the Dirichlet series 
$$
\sum_{n=0}^{\infty}a(p^{j(1+2 n)})\frac{1}{p^{s j(1+2 n)}}
$$ 
can be meromorphicaly extended to the whole complex plane $\mathbb{C}$. Suppose that the sequence $(a(p^{j(1+2 n)}))_{n\in\mathbb{N}}$ does not have infinitely many sign changes for infinitely many primes $p$ and apply once again Landau's theorem.

Suppose now that one of the denominators on the right-hand side of \eqref{eq,54} has a real zero, for example $\alpha_{p}\zeta^{\mu}\in\R$. Then as in the proof of \thref{thm2} we find
$$
a(p)=(\alpha_p+\beta_p)\chi_{0}(p)=\pm p^{(k-1)/2}(\zeta^{-\mu}+\zeta^{\mu}).
$$
We repeat the procedure of \thref{thm2} to show that the right-hand side of \eqref{eq,54} has no real poles, and then the first case of Landau's theorem is excluded.   

It remains to exclude the second case of Landau's theorem. By \thref{lem}, we have  

\begin{align}
S_{1,j}(X)&=\sum_{n=0}^{\infty}a(p^{j n})X^{j n}-\sum_{n=0}^{\infty}a(p^{2 j n})X^{2jn}\\\label{fin}
&=\frac{p^{2 j (k-1)}X^{4j}-\left( a(p^{2 j})+p^{j(k-1)}\right) X^{2j}-a(p^{j})X^{j}}{\left(1+ a(p^{j})X^{j}+p^{j(k-1)}X^{2j}\right)\left(1-a(p^{2 j})X^{2j}+p^{2 j(k-1)}X^{4j}\right)}\cdot
\end{align}

The numerator on the right-hand side is a polynomial of degree $4j$ and the denominator is a non constant polynomial of degree $6j$, hence the denominator has zeros. Setting $X=p^{-s}$ to obtain a contradiction.

\end{proof}

\section{Proof of \texorpdfstring{\thref{thm4}}{Theorem 4}}

\begin{proof}

The proof is similar to the one of \cite[Theorem 2.2]{kohnen14}, it suffices to make the following change, the set $V_{p}\subseteq S_{k}$ is defined to be the set of all cusp forms $g$ whose Fourier coefficients $b(p^{j(1+2n)})$ satisfy $b(p^{j(1+2n)})\ll_{g,c}p^{jc(1+2n)}$ for all $n\geq 0$ and every $c\in\mathbb{R}$. The first part of the proof remains unchanged. Now, $V_p$ is stable under $T_{j}(p)^{2}$, then by the same argument, there is an eigenform $f_{0}\in V_p$ of $T_{j}(p)^{2}$ since this operator is Hermitian. 

From this and since $P(T_{j}(p),X)$ is irreducible, we deduce that there is $\lambda\neq 0$ such that $T_{j}(p)^{2}f_{0}=\lambda f_{0}$. We should note that $T_{j}(p)h_{1}=\sqrt{\lambda}h_{1}$ and $T_{j}(p)h_{2}=-\sqrt{\lambda}h_{2}$ where $h_1=\sqrt{\lambda}f_{0}+T_{j}(p)f_{0}$ and $h_2=-\sqrt{\lambda}f_{0}+T_{j}(p)f_{0}$. Then by our hypothesis, either $h_1=0$ or $h_2=0$. Suppose without loss of generality that $h_2=0$ and $T_{j}(p)f_{0}=\sqrt{\lambda}f_{0}$. We can now proceed as in the proof of \cite[Theorem 2.2]{kohnen14} to deduce that $f_0$ is an eigenfunction of all Hecke operators. Finally we apply \eqref{fin} to $f_0$ and the second case of Landau's theorem is excluded.
\end{proof}

\section{Sign changes of the sequence \texorpdfstring{$\left(\frac{a(p^{l+m_{p}n})}{\chi_{0}(p)^{l}}\right)_{n\in\N}$}{TEXT}}

Finally, by modifying the method above one can obtain the following result. 

\begin{thm}\thlabel{thm5}
Let $f\in S_k(\Gamma_{0}(N),\chi)$ be a normalized Hecke eigenform of even integral weight $k$ and level $N$, with Dirichlet character $\chi$. Let $\chi_{0}\pmod N$  be a Dirichlet character satisfying $\chi_{0}^{2}=\chi$. Let
$$
f(z)=\sum_{n\ge 1}a(n)e(nz),
$$
be the Fourier expansion of $f$ at $\infty$. Consider the primes $p$ for which the polynomial
$(\beta_{p}\alpha_{p}^{m_{p}}-\alpha_{p}\beta_{p}^{m_{p}})X^{m_{p}}+(\beta_{p}^{m_{p}}-\alpha_{p}^{m_{p}})X^{m_{p}-1}+(\alpha_{p}-\beta_{p})$ has no real zero, where $m_{p}$ is an integer satisfying $\chi_{0}(p)^{m_{p}}=1$. Then for almost all of those primes $p$, the sequence $\left(\frac{a(p^{l+m_{p}n})}{\chi_{0}(p)^{l}}\right)_{n\in\N}$ has infinitely many sign changes with $l$ runs through the integers satisfying $1\leq l\leq m_{p}-1$.
\end{thm}

\begin{rmk}
Notice that for those sequences, $\left(Re(a(p^{l+m_{p}n}))\right)_{n\in\N}$ \\
(resp. $\left(Im(a(p^{l+m_{p}n}))\right)_{n\in\N}$) has infinitely many sign changes when $\chi_{0}(p)^{l}\neq \pm i$ (resp. $\chi_{0}(p)^{l}= \pm i$).
\end{rmk}

Before giving the proof we shall establish some needed formulas in the full generality. Assume the conditions of \thref{thm2}. Let $\omega:=e^{2\pi i/m}$ be a primitive $m-$th root of unity of order $m$ and $\chi_{0}(p)^{m}=1$. We want  to compute the following sum
$$
S_{l}=\sum_{n=0}^{\infty}\frac{a(p^{l+m n})}{\chi_{0}(p^{l+m n})}X^{l+m n},
$$
where $l$ is an integer satisfying $0\leq l\leq m-1$. By \eqref{eq,99}, we have
$$
a(p^{l+m n})=a(p)a(p^{l-1+m n})-p^{k-1}\chi(p)a(p^{l-2+m n}),
$$
this yields
\begin{equation}
S_{l}=\frac{a(p)}{\chi_{0}(p)}S_{l-1}X-p^{k-1}X^{2}S_{l-2}.\label{eq,41}
\end{equation}
On the other hand, we have
\begin{equation}
S_{0}+\dots +S_{m-1}=P=\sum_{n\ge0}\frac{a(p^{n})}{\chi_{0}(p^{n})}X^n=\dfrac{1}{1-\frac{a(p)}{\chi_{0}(p)}X+p^{k-1}X^2}\cdot\label{eq,42a}
\end{equation}

From \eqref{eq,41}, we get 
\begin{equation}
S_{l}=(a\alpha_{p}^{l}+b\beta_{p}^{l})X^{l},\label{eq,43}
\end{equation}
where $a$ and $b$ are terms depending upon $X$ which will be computed.  

By the same reasoning as in \eqref{eq:5} we have

\begin{equation}
S_{0}=\sum_{n\ge0}\frac{a(p^{m n})}{\chi_{0}(p^{m n})}X^{m n}=\frac{1}{m}\sum_{\mu=0}^{m-1}\dfrac{1}{(1-\omega^{\mu}\alpha_p X)(1-\omega ^{\mu}\beta_p X)}\cdot\label{eqq,1}
\end{equation}

Hence by \eqref{eq,43}, we have

\begin{equation}
a +b=S_{0}=\frac{1}{m}\sum_{\mu=0}^{m-1}\dfrac{1}{(1-\omega^{\mu}\alpha_p X)(1-\omega^{\mu}\beta_p X)}\cdot\label{eq,46}
\end{equation}

Combine now the equations \eqref{eq,42a} and \eqref{eq,43} to get

\begin{align}
S_{0}+\cdots +S_{m-1}&=a\sum_{l=0}^{m-1}\alpha_{p}^{l}X^{l}+b\sum_{l=0}^{m-1}\beta_{p}^{l}X^{l}\cdot\\
&=a \frac{\alpha _{p}^{m}X^{m}-1}{\alpha _{p}X-1}+b\frac{\beta_{p}^{m}X^{m}-1}{\beta_{p}X-1}\cdot\\
&=P.\label{eq,47}
\end{align}

From this and \eqref{eq,43} we obtain

\begin{equation}
R(X)a=P-S_{0}\frac{\beta_{p}^{m}X^{m}-1}{\beta_{p}X-1},\label{eq,48}
\end{equation}

\begin{equation}
R(X)b=S_{0}\frac{\alpha_{p}^{m}X^{m}-1}{\alpha_{p}X-1}-P,\label{eq,49}
\end{equation}

where 

\begin{gather*}
R(X)=\left(\frac{\alpha_{p}^{m}X^{m}-1}{\alpha_{p}X-1}\right)-\left(\frac{\beta_{p}^{m}X^{m}-1}{\beta_{p}X-1}\right).
\end{gather*}

Replacing this in \eqref{eq,43}, then

\begin{align}
S_{l}&=\frac{\alpha_{p}^{l}X^{l}\left(P-S_{0}\frac{\beta_{p}^{m}X^{m}-1}{\beta_{p}X-1}\right)+\beta_{p}^{l}X^{l}\left(S_{0}\frac{\alpha_{p}^{m}X^{m}-1}{\alpha_{p}X-1}-P\right)}{R(X)}\\
&=\frac{\alpha_{p}^{l}X^{l}\left(P(\beta_{p}X-1)-S_{0}(\beta_{p}^{m}X^{m}-1)\right)+\beta_{p}^{l}X^{l}\left(S_{0}(\alpha_{p}^{m}X^{m}-1)-P(\alpha_{p}X-1)\right)}{(\beta_{p}\alpha_{p}^{m}-\alpha_{p}\beta_{p}^{m})X^{m}+(\beta_{p}^{m}-\alpha_{p}^{m})X^{m-1}+(\alpha_{p}-\beta_{p})}\cdot\label{eq,50}
\end{align}

Notice that using \eqref{eq,50}, the Dirichlet series 
$$
\sum_{n=0}^{\infty}\frac{a(p^{l+m n})}{\chi_{0}(p^{l+m n})}p^{-s(l+m n)}
$$ 
can be meromorphicaly extended to the whole complex plane $\mathbb{C}$.

\begin{proof}[Proof of \thref{thm5}]

Suppose that the sequence $\left(\frac{a(p^{l+m n})}{\chi_{0}(p^{l+m n})}\right)_{n\geq 0}$ does not have infinitely many sign
changes.

Applying Landau's theorem, we deduce that the Dirichlet series
\begin{equation}
\sum_{n=0}^{\infty}\frac{a(p^{l+m n})}{\chi_{0}(p^{l+m n})}p^{-s(l+m n)}\quad (\Re(s)\gg 1),\label{eq:2a}
\end{equation}
either has a pole on the real point of its line of convergence or must converges for all $s\in\C$. We start by considering the first case.

Since the denominator of \eqref{eq,50} has no real pole by hypothesis, then either the denominator of $P(X)$ or one of the denominators of \eqref{eqq,1} has real zero. We deduce that in all cases $\sqrt{p}\in\mathbb{K}$. The contradiction is obtained by the same way as above. Consequently, for almost all primes $p$ satisfying the hypothesis, the right-hand side of \eqref{eq:2a} has no real poles. We exclude the second case of Landau's theorem by using the both equations \eqref{eq,50} and \eqref{eeqq1}.

\end{proof}

\bibliographystyle{spmpsci}
\bibliography{mybibfile}

\end{document}